\documentclass[reqno,oneside]{amsart}

\usepackage{algorithm}
\usepackage{algorithmic}
\usepackage{amsfonts}
\usepackage{amsmath}
\usepackage{amssymb}
\usepackage{amsthm}
\usepackage{bm}
\usepackage{cite}
\usepackage{geometry}
\usepackage{graphicx}
\usepackage{hyperref}

\newcommand{\cA}{\mathcal A}
\newcommand{\f}{\frac}

\newcommand{\p}{\partial}

\newcommand{\be}{\begin{equation}}
\newcommand{\ee}{\end{equation}}
\newcommand{\ba}{\begin{array}}
\newcommand{\ea}{\end{array}}

\theoremstyle{plain}
\newtheorem{theorem}{Theorem}[section]
\newtheorem{lemma}{Lemma}[section]
\newtheorem{problem}{Problem}[section]

\title[Identifying the source term in a viscoelastic membrane]
{Identifying the source term in a viscoelastic membrane with a Riemann-Liouville time derivative by the partial interior observation}

\author[Z. Yang]{Zhiwei Yang$^{1}$}
\thanks{$^1$School of Qilu Transportation and State Key Laboratory of Intelligent Manufacturing of Advanced Construction Machinery, Shandong University, Jinan, Shandong 250002, P.R. China. E-mail: {\tt zhiweiyang@sdu.edu.cn}
}

\author[Y. Liu]{Yikan Liu$^{2,*}$}
\thanks{$^2$ Department of Mathematics, Kyoto University, Kitashirakawa-Oiwakecho, Sakyo-ku, Kyoto 606-8502, Japan.}
\thanks{$^*$Corresponding author. E-mail: {\tt liu.yikan.8z@kyoto-u.ac.jp}}

\begin{document}

\begin{abstract}
This paper studies an inverse source problem for a viscoelastic membrane, where the material's memory effect is characterized by the Riemann-Liouville fractional derivative. The problem is to recover the unknown source term from the limited interior observation data. We propose an optimal control framework to address this ill-posed inverse problem. The first-order optimality condition leads to a coupled system of forward and backward fractional partial differential equations. A numerical algorithm combining the finite element method and a conjugate gradient iterative scheme is then developed for the reconstruction of the source term. Several numerical examples are provided to demonstrate the effectiveness and robustness of the proposed method.
\end{abstract}

\maketitle


\section{Introduction}

The vibration analysis of beams and membranes has garnered significant attention due to its wide range of applications in various engineering fields \cite{AnsEsm,BookInman,BookReddy2003,BookTimo,BookWinkler,Mei}. Understanding the dynamic behavior of these structural components is essential for designing systems with optimal performance and durability. This is particularly true for viscoelastic materials, which exhibit a combination of elastic (energy storage) and viscous (energy dissipation) behavior, making them crucial in applications where dynamic response and damping are important \cite{Kerr,FilBor,Chr,Fer,Gem,PraPlate20,Vla}. The study of viscoelastic membrane vibrations is therefore highly relevant to numerous disciplines, including aerospace, automotive, and civil engineering.

Classical models often rely on integer-order differential equations to describe material behavior. However, recent advances have demonstrated the superiority of fractional-order derivatives in capturing the complex, history-dependent dynamics inherent to viscoelastic materials \cite{Bag,BagTor,BonKap,JaiMcKin,Mai,ZheWanCNS}. The fractional calculus has consequently gained popularity as an effective tool in diverse engineering domains. Its ability to accurately describe long-term memory effects and anomalous diffusion processes is especially valuable for predicting long-term durability and stability \cite{LiWan,Pip,Pao,Rei,Sun,SunGao,SunZha,VarKha}. By incorporating fractional-order derivatives into constitutive models, engineers can achieve more precise simulations of viscoelastic materials, leading to improved design strategies and enhanced structural performance \cite{MaiSpa,WeiChe,YinLi,YouMaj,YouHos}.

The analysis of structures involving viscoelastic materials presents considerable challenges, primarily due to their complex constitutive relationships and the resulting intricate dynamic response \cite{PerKar,PraPlate22,PraBeam,Ross}. In recent years, significant research efforts have been directed towards parameter identification and inverse problems for damped structural systems, such as beams and plates \cite{HasIPplate,HasRom,HasTrans,HasIPbeam}. Building upon these foundations, this paper focuses on a specific inverse source problem related to the vibrational behavior of viscoelastic membranes. We aim to elucidate the intricate interplay between elasticity and viscosity by investigating the following problem: recovering the spatial load component \( q(\bm x) \) in the viscoelastic membrane equation

Bearing the above background in mind, we formulate the problem under consideration in this article as follows. Let $T>0$ be a constant and $\Omega\subset\mathbb R^d$ ($d=1,2,\dots$) be a bounded domain with a smooth boundary $\p\Omega$. Define an elliptic operator $\cA:D(\cA)\longrightarrow L^2(\Omega)$ as
\begin{align*}
\cA\varphi(\bm x) & := -\mathrm{div}(\bm A(\bm x)\nabla\varphi(\bm x))+\bm b(\bm x)\cdot\nabla\phi(\bm x)+c(\bm x)\varphi(\bm x),\\
D(\cA) & := \{\varphi\in H^2(\Omega)\mid\p_{\bm A}\varphi+\sigma\varphi=0\mbox{ on }\p\Omega\}.
\end{align*}
where $\cdot$ and $\nabla=(\f\p{\p x_1},\dots,\f\p{\p x_d})$ denote the inner product in $\mathbb R^d$ and the gradient in $\bm x$, respectively. Here $\bm A\in C^1(\overline\Omega;\mathbb R^{d\times d})$ is a strictly positive-definite matrix-valued function on $\overline\Omega$ satisfying
\[
\bm A(\bm x)\bm\xi\cdot\bm\xi\ge\kappa|\bm\xi|^2=\kappa\,\bm\xi\cdot\bm\xi,\quad\forall\,\bm x\in\overline\Omega,\ \forall\,\bm\xi\in\mathbb R^d
\]
for some constant $\kappa>0$. Further, we assume $\bm b\in C^1(\overline\Omega;\mathbb R^d)$, $c\in C(\overline\Omega)$, $0\le\sigma\in C(\p\Omega)$, and
\[
\p_{\bm A}\varphi(\bm x):=\bm A(\bm x)\nabla\varphi(\bm x)\cdot\bm\nu(\bm x)
\]
stands for the normal derivative of $\varphi$ on $\p\Omega$ associated with $\bm A$, where $\bm\nu(\bm x)$ is the unit outward normal vector to $\p\Omega$ at $\bm x\in\p\Omega$.

Next, given a constant $\beta>0$, we introduce the Riemann-Liouville integral operator of order $\beta$ in time as
\[
J_{0+}^\beta f(t):=\int_0^t\f{(t-s)^{\beta-1}}{\Gamma(\beta)}f(s)\,\mathrm ds,\quad f\in C[0,\infty),
\]
where $\Gamma(\,\cdot\,)$ denotes the Gamma function. Then for a constant $\alpha\in(1,2)$, the Riemann-Liouville derivative of order $\alpha$ is defined by (see e.g. \cite{Die,MeeSik,Pod,SamKil})
\begin{equation}\label{def:frac}
D_{0+}^\alpha:=\f{\mathrm d^2}{\mathrm dt^2}\circ J_{0+}^{2-\alpha}:C^2[0,\infty)\longrightarrow C[0,\infty),
\end{equation}
where $\circ$ stands for the composite. Now we are well prepared to formulate the initial-boundary value problem for a nonlocal viscoelastic equation as
\begin{equation}\label{main:e0}
\begin{cases}
(\eta(\bm x)\partial_{t}^2 + \mu(\bm x)D_{0+}^{\alpha}u + \cA) u(\bm x,t) = F(\bm x,t) := p(t)q(\bm x), & (\bm x,t) \in \Omega \times (0,T),\\
u(\bm x,0) = u_0(\bm x), \ \partial_{t}u(\bm x,0) = u_1(\bm x), & \bm x \in \Omega,\\
\p_{\bm A} u(\bm x,t)+ \sigma u(\bm x,t) = 0, & (\bm x,t) \in \partial \Omega \times (0,T),
\end{cases}
\end{equation}
where $\eta,\mu\in C(\overline\Omega)$ are strictly positive on $\overline\Omega$, and $u_0\in H^1(\Omega),u_1\in L^2(\Omega)$ denote initial displacement and initial velocity, respectively. In particular, the external load $F$ is assumed to take the form of separated variables, where \( p(t) \) is the known temporal component and \( q(\bm x) \) is the unknown spatial component to be identified. The boundary condition covers both homogeneous Neumann and Robin cases, depending on whether $\sigma$ vanishes identically on $\p\Omega$ or not. The discussion throughout this article also works for the homogeneous Dirichlet boundary condition.

This manuscript is concerned with the following inverse problem for \eqref{main:e0}.

\begin{problem}\label{prob-ISP}
Let $\Omega_0$ be a nonempty subdomain of $\Omega$ and $u$ be the solution to \eqref{main:e0}. Determine the spatial component $q(\bm x)$ in the source term of \eqref{main:e0} by the partial interior observation of $u$ in $\Omega_0\times(0,T)$.
\end{problem}

To address this inverse problem, we reformulate it as a minimization problem using a least-squares cost functional. Within an optimal control framework, we derive a coupled forward-backward system of differential equations from the first-order optimality conditions. Subsequently, we develop a numerical algorithm based on the conjugate gradient method, combined with a linear-type finite element discretization, for stable and efficient reconstruction of the source term \( q(\bm x) \).

The remainder of this paper is organized as follows. Section \ref{sect:model} mentions the physical backgrounds of the governing equations from a membrane vibration problem and a fractional spring-dashpot model. Then the development of a fully discrete finite element scheme for the direct problem is presented in Sections \ref{sect:num}. The optimization approach to solving the inverse problem is detailed in Section \ref{sect:inverse}. Several numerical experiments are provided in Section \ref{sect:ex} to validate the numerical analysis and demonstrate the effectiveness of the proposed inversion algorithm. Finally, discussions and concluding remarks are given in the last section.


\section{Preliminaries}


\subsection{Model formulation}\label{sect:model}

This subsection briefly explains the motivation of studying the governing equation in \eqref{main:e0} from practical applications.

First we consider a membrane with the mid-plane domain \(\Omega \subset\mathbb R^2\) in its undeformed state. Let \(\rho(\bm x)\) be the mass density, \(h\) the thickness, and \(u(\bm x,t)\) the transverse displacement. The membrane is subjected to a transverse load decomposed as \(F(\bm x,t)\). In the classical Winkler foundation model, the foundation is modeled as a set of linearly elastic springs, yielding a restoring force proportional to the displacement:
\begin{equation}\label{force:pasnak}
F_r(\bm x,t) = \mu(\bm x) u(\bm x,t),
\end{equation}
where \(\mu(\bm x)\) is the stiffness coefficient. Then the governing equation for the membrane on such a foundation is given by (see \cite{BookRao,BookReddy2006})
\[
\rho(\bm x) h \, \partial_{t}^2 u - \mathrm{div}(C(\bm x) \nabla u) + \mu(\bm x) u = F(\bm x,t)\quad\mbox{in }\Omega \times (0,T),
\]
where \(C(\bm x)\) denotes the tension in the plane. The system is closed with appropriately specified initial and boundary conditions such as those in \eqref{main:e0}.

However, real-world foundation materials often exhibit viscoelastic behavior that involves memory-dependent damping. To accurately capture this behavior, we adopt the Riemann-Liouville fractional derivative, as introduced in \eqref{def:frac}, to model the reaction force of the foundation as
\begin{equation}\label{gWinklerForce}
F_r(\bm x,t) = \mu(\bm x)D_{0+}^{\alpha} u, \quad \alpha \in (1,2),
\end{equation}
where \(\mu(\bm x)\) now stands for the viscoelastic coefficient. Replacing the elastic force \eqref{force:pasnak} with the viscoelastic model \eqref{gWinklerForce}, we obtain the governing equation for a membrane on a fractional viscoelastic foundation
\begin{equation}\label{frac:PlateEq}
\rho(\bm x) h \, \partial_{t}^2 u + \mu(\bm x)D_{0+}^{\alpha} u - \mathrm{div}(C(\bm x) \nabla u) = F(\bm x,t)\quad\mbox{in }\Omega \times (0,T).
\end{equation}

On the other hand, the following double-term time-fractional wave equation was derived from a one-dimensional fractional extended Maxwell model in \cite{KLLNY} very recently:
\begin{equation}\label{eq-Maxwell}
\rho(x)\p_t^2u+\rho(x)\lambda(x)D_{0+}^\alpha u-\p_x(C(x)\p_x u)=F(x,t)\quad\mbox{in }\Omega\times(0,T),
\end{equation}
where takes rather a similar form as \eqref{gWinklerForce}. As a reasonable and natural generalization of \eqref{gWinklerForce} and \eqref{eq-Maxwell}, in this paper we fix \eqref{main:e0} as the model problem for investigations.


\subsection{Numerical discretization of the forward problem}\label{sect:num}

The accurate resolution of the forward problem is paramount for the subsequent solution of the inverse problem. The principal challenge in numerically integrating equation \eqref{frac:PlateEq} stems from the nonlocal character of the Riemann-Liouville fractional derivative, which necessitates a discretization scheme that faithfully captures the inherent memory effects. This section delineates the development of a robust and high-fidelity numerical scheme, integrating a temporally stable finite-difference discretization with a spatially flexible finite element method.

We start with discretizing the temporal interval $[0, T]$, which is partitioned into $N$ uniform segments, defining a sequence of time steps:
\begin{equation}\label{Temp:e0}
t_n = n \tau, \quad \text{for } n = 0, 1, \ldots, N, \quad \tau = T/N.
\end{equation}
At a discrete time $t_n$, the governing equation is expressed as
\begin{equation}\label{PlateEqL}
\eta(\bm x)\p_t^2 u(\bm x,t_n) + \mu(\bm x)D_{0+}^\alpha u(\bm x,t_n) + \cA u(\bm x,t_n) = F(\bm x,t_n).
\end{equation}
The integer-order time derivatives are approximated using finite differences renowned for their numerical stability in dynamic problems. The first and second derivatives are discretized as follows:
\begin{align}
\partial_t u(\bm x,t_n) & \approx \delta_{\tau} u(\bm x,t_n) := \frac{u(\bm x,t_n) - u(\bm x,t_{n-1})}{\tau}, \label{Temp:e11} \\
\partial_t^2 u(\bm x,t_n) & \approx \delta_{\tau}^2 u(\bm x,t_n) := \frac{u(\bm x,t_n) - 2u(\bm x,t_{n-1}) + u(\bm x,t_{n-2})}{\tau^2}. \label{Temp:e2}
\end{align}

Next, the accurate discretization of the fractional term $D_t^\alpha u$ is critical. We employ the Gr\"unwald-Letnikov formula, which provides a consistent approximation that converges to the Riemann-Liouville derivative. The discrete form reads
\begin{equation}\label{disc:frac1}
D_{0+}^\alpha u(\bm x,t_n) \approx \delta^{\alpha}_{\tau} u(\bm x,t_n) := \tau^{-\alpha} \sum_{k=0}^{n} \omega_k^{(\alpha)} u(\bm x,t_{n-k}),
\end{equation}
where the coefficients $\omega_k^{(\alpha)}$ are generated recursively to efficiently account for the entire history of the solution:
\[
\omega_0^{(\alpha)} = 1, \quad \omega_k^{(\alpha)} = \left(1 - \frac{\alpha+1}{k}\right) \omega_{k-1}^{(\alpha)}, \quad k \geq 1.
\]
This scheme inherently embodies the hereditary nature of viscoelasticity, making it physically congruent with our model.

On the spatial direction, to accommodate the complex geometry of the membrane domain $\Omega$, we derive a variational formulation. Let $\chi \in H^1(\Omega)$ be a test function. Multiplying both sides of \eqref{PlateEqL} by $\chi$ integrating over $\Omega$, we apply the divergence theorem and the boundary condition to obtain the weak form:
\[
\int_\Omega \left(\eta \, \partial_t^2 u + \mu \, D_{0+}^\alpha u\right) \chi \,\mathrm d\bm x + A[u, \chi] + \int_{\p\Omega}\sigma u\chi\,\mathrm d\bm S= \int_\Omega F \, \chi \,\mathrm d\bm x,
\]
where the bilinear form $A[\,\cdot\,, \,\cdot\,]$, representing the internal elastic energy, is defined as
\[
A[\varphi, \chi] := \int_\Omega\{\bm A\nabla\varphi \cdot \nabla \chi+(\bm b\cdot\nabla\varphi+c\,\varphi)\}\chi\,\mathrm d\bm x.
\]
Substituting the temporal approximations \eqref{Temp:e11}--\eqref{disc:frac1} into the weak form leads to the semi-discrete problem. For the spatial discretization, the domain $\Omega$ is triangulated into a quasi-uniform mesh. We construct a finite-dimensional subspace $S_h \subset H^1_0(\Omega)$ comprising continuous piecewise linear polynomials. The fully discrete problem is then formulated as follows: for $n \geq 2$, find $u_h^n \in S_h$ such that $\forall \chi_h \in S_h$,
\[
\int_\Omega \left( \eta \, \delta_{\tau}^2 u_h^n + \mu \, \delta_{\tau}^{\alpha} u_h^n \right) \chi_h \,\mathrm d\bm x + A[u_h^n, \chi_h] = \int_\Omega F^n \, \chi_h \,\mathrm d\bm x.
\]

Finally, the initial conditions have be projected onto the finite element space to ensure compatibility. We utilize the Ritz projection $\Pi_h: H^1_0(\Omega) \to S_h$ (see \cite{CiaBook}), defined uniquely by:
\[
A[\Pi_h v, \chi_h = A[v, \chi_h], \quad \forall \chi_h \in S_h.
\]
The discretized initial data are then set as:
\begin{equation}\label{FEM:Init}
u_h^0 = \Pi_h u_0, \quad u_h^1 = \Pi_h (u_0 + \tau u_1).
\end{equation}
This choice preserves the optimal convergence rates. The resulting linear system at each time step is symmetric and positive definite, guaranteeing the existence of a unique solution $u_h^n$ that can be efficiently computed.


\section{Optimization formulation of the inverse source problem}\label{sect:inverse}

Within this section, we address the critical task of reconstructing the unknown spatial source term $q(\bm x)$ from partial interior measurement data. To this end, we employ the framework of optimal control theory, which provides a powerful and systematic methodology for solving such ill-posed inverse problems. The core idea is to reformulate the inverse problem as a minimization problem, where the goal is to find a source term that produces a solution to the forward model that best matches the observed data, while incorporating appropriate regularization to ensure the stability and well-posedness of the solution procedure.

We now present the detailed optimal control formulation for the regularized inverse source problem. We define the regularized form of the Tikhonov functional by extending it as follows:
\begin{equation}\label{ocp:e0}
\mathcal{J}(q):=\frac{1}{2}\|u(q)-u_\delta\|_{L^2(\Omega_0\times(0,T))}^2 + \frac{\beta}{2}\|q\|_{L^2(\Omega)}^2,\quad q\in L^2(\Omega).
\end{equation}
In this formulation, $u(q)$ denotes the solution to the initial-boundary value problem \eqref{main:e0} with the source term $q$, which serves as the control variable. The function $u_\delta$ represents the noisy measured displacement data available only in the subdomain $\Omega_0 \subset \Omega$ over the time interval $(0, T)$, where $\delta>0$ stands for the noisy level. The first term in the functional $\mathcal{J}(q)$ is the misfit term, which quantifies the discrepancy between the computed solution and the measured data. The second term in $\mathcal{J}(q)$ is the Tikhonov regularization term, where $\beta>0$ is the regularization parameter. This term is essential for stabilizing the inversion process by penalizing large oscillations in the solution and ensuring the continuous dependence of the solution on the data, thereby mitigating the ill-posedness inherent in the inverse problem.

Subsequently, the inverse source problem is rigorously reformulated as the following constrained optimal control problem: determine a source function $q^* \in L^2(\Omega)$ that minimizes the cost functional $\mathcal{J}(q)$ over $L^2(\Omega)$, under the constraint that the state variable $u(q)$ satisfies the governing fractional viscoelastic membrane equation \eqref{main:e0}. This formulation establishes a solid theoretical foundation for the subsequent numerical solution algorithm.

Having established the optimal control framework, we now turn to the basic theoretical question on the unique existence of the solution to the inverse source problem. To this end, we introduce the adjoint $\cA^*$ of the elliptic operator $\cA$ satisfying
\begin{equation}\label{eq-adjoint}
(\cA\varphi,\psi)=(\varphi,\cA^*\psi),\quad\forall\,\varphi\in D(\cA),\ \forall\,\psi\in D(\cA^*),
\end{equation}
where $(\,\cdot\,,\,\cdot\,)$ denotes the inner product of $L^2(\Omega)$. Then it follows immediately from direct calculation that
\[
\cA^*\psi=-\mathrm{div}(\bm A\nabla\psi+\bm b\psi)+c\psi,\quad D(\cA^*)=\{\psi\in H^2(\Omega)\mid\p_{\bm A}\psi+(\sigma+\bm b\cdot\bm\nu)\psi=0\mbox{ on }\p\Omega\}.
\]
On the temporal direction, we introduce the backward Riemann-Liouville integral operator of order $\beta>0$ as well as the backward Caputo derivative of order $\alpha\in(1,2)$ as
\begin{align*}
J_{T-}^\beta f(t) & :=\int_t^T\f{(s-t)^{\beta-1}}{\Gamma(\beta)}f(s)\,\mathrm d s,\quad f\in C(-\infty,T],\\
\mathrm d_{T-}^\alpha & :=J_{T-}^{2-\alpha}\circ\f{\mathrm d^2}{\mathrm d t^2}:C^2(-\infty,T]\longrightarrow C(-\infty,T].
\end{align*}
Now we have collected all necessary ingredients to state the main theorem of this article.

\begin{theorem}[Gradient of the regularized cost functional]\label{thm:gradient}
The Fr\'echet derivative of the Tikhonov-regularized cost functional $\mathcal{J}(q)$ defined in \eqref{ocp:e0} with respect to the source parameter $q \in L^2(\Omega)$ is given by the following explicit formula:
\begin{equation}\label{thmIP:e0}
\mathcal{J}'(q) = \int_0^T p(t)v(q)(\,\cdot\,,t)\,\mathrm dt + \beta\,q,
\end{equation}
where the adjoint state variable $v(q)$ satisfies the backward problem
\begin{equation}\label{eq-gov-v}
\begin{cases}
(\eta\,\p_t^2+\mu\,\mathrm d_{T-}^\alpha+\cA^*)v=\bm1_{\Omega_0}(u(q)-u_\delta) & \mbox{in }\Omega\times(0,T),\\
v=\p_t v=0 & \mbox{in }\Omega\times\{T\},\\
\p_{\bm A}v+(\sigma+\bm b\cdot\bm\nu)v=0 & \mbox{on }\p\Omega\times(0,T),
\end{cases}
\end{equation}
where $\bm1_{\Omega_0}$ denotes the characteristic function of $\Omega_0$.
\end{theorem}

The proof of the above theorem relies on the following key lemma on the fractional integration by parts. Similar formulae were derived in \cite{L21}, whereas we still provide a verification for completeness.

\begin{lemma}[Fractional integration by parts]\label{lem-FIbP}
Let $\alpha\in(1,2)$ and $f,g\in C^2[0,T]$. Then
\[
\int_0^T(D_{0+}^\alpha f)g\,\mathrm d t=\Big[(J_{0+}^{2-\alpha}f)'g-(J_{0+}^{2-\alpha}f)g'\Big]_0^T+\int_0^T f(\mathrm d_{T-}^\alpha g)\,\mathrm d t.
\]
\end{lemma}

\begin{proof}
By the definition of $D_{0+}^\alpha$, we employ the usual integration by parts to calculate
\begin{align*}
\int_0^T(D_{0+}^\alpha f)g\,\mathrm d t & =\int_0^T(J_{0+}^{2-\alpha}f)''g\,\mathrm d t=\Big[(J_{0+}^{2-\alpha}f)'g\Big]_0^T-\int_0^T(J_{0+}^{2-\alpha}f)'g'\,\mathrm d t\\
& =\Big[(J_{0+}^{2-\alpha}f)'g-(J_{0+}^{2-\alpha}f)g'\Big]_0^T+\int_0^T(J_{0+}^{2-\alpha}f)g''\,\mathrm d t.
\end{align*}
For the last term above, we recall the following formula relating the forward and backward Riemann-Liouville integral operators (see \cite[Lemma 4.1]{JLLY17}):
\[
\int_0^T(J_{0+}^\beta f)g\,\mathrm d t=\int_0^T f(J_{T-}^\beta g)\,\mathrm d t,\quad\beta>0,\ f,g\in C[0,T].
\]
Then we obtain
\[
\int_0^T(J_{0+}^{2-\alpha}f)g''\,\mathrm d t=\int_0^T f(J_{T-}^{2-\alpha}g'')\,\mathrm d t=\int_0^T f(\mathrm d_{T-}^\alpha g)\,\mathrm d t
\]
by the definition of $\mathrm d_{T-}^\alpha$, which completes the proof.
\end{proof}

Now we can proceed to the proof of Theorem \ref{thm:gradient}.

\begin{proof}[Proof of Theorem $\ref{thm:gradient}$]
First we calculate the G\^ateaux derivative of $\mathcal{J}$ at $q\in L^2(\Omega)$ in the direction $\widetilde q\in L^2(\Omega)$. To this end, we pick a small $\varepsilon>0$ to calculate
\begin{equation}\label{eq-Gateaux}
\f{\mathcal J(q+\varepsilon\widetilde q)-\mathcal J(q)}\varepsilon=\int_0^T\!\!\!\int_{\Omega_0}w(\widetilde q)\left(\f{u(q+\varepsilon\widetilde q)+u(q)}2-u_\delta\right)\mathrm d\bm x\mathrm d t+\beta\int_\Omega\widetilde q\left(q+\f\varepsilon2\widetilde q\right)\mathrm d\bm x,
\end{equation}
where we put
\[
w(\widetilde q):=\f{u(q+\varepsilon\widetilde q)-u(q)}\varepsilon.
\]
Taking difference between the initial-boundary value problems of $u(q+\varepsilon\widetilde q)$ and $u(q)$, we see that $w(\widetilde q)$ satisfies
\begin{equation}\label{eq-gov-w}
\begin{cases}
(\eta\,\p_t^2+\mu\,D_{0+}^\alpha+\cA)w=p\,\widetilde q & \mbox{in }\Omega\times(0,T),\\
w=\p_t w=0 & \mbox{in }\Omega\times\{0\},\\
\p_{\bm A}w+\sigma w=0 & \mbox{on }\p\Omega\times(0,T).
\end{cases}
\end{equation}
Meanwhile, we have $u(q+\varepsilon\widetilde q)\longrightarrow u(q)$ in $L^2(\Omega_0\times(0,T))$ as $\varepsilon\to0$ by the continuity of the solution to \eqref{main:e0} with respect to the source term. Then we pass $\varepsilon\to0$ in \eqref{eq-Gateaux} to deduce the G\^ateaux derivative as
\begin{equation}\label{eq-Gateaux1}
\mathcal J'(q)\widetilde q=\int_0^T\!\!\!\int_{\Omega_0}w(\widetilde q)(u(q)-u_\delta)\,\mathrm d\bm x\mathrm d t+\beta\int_\Omega q\,\widetilde q\,\mathrm d\bm x.
\end{equation}

To further acquire the Fr\'echet derivative $\mathcal J'(p)$, we recall the backward problem \eqref{eq-gov-v} and rewrite the first term on the right-hand side of \eqref{eq-Gateaux1} as
\begin{equation}\label{eq-Gateaux2}
\int_0^T\!\!\!\int_{\Omega_0}w(\widetilde q)(u(q)-u_\delta)\,\mathrm d\bm x\mathrm d t=\int_0^T\!\!\!\int_\Omega w(\widetilde q)\bm1_{\Omega_0}(u(q)-u_\delta)\,\mathrm d\bm x\mathrm d t=I_1+I_2+I_3,
\end{equation}
where
\begin{gather*}
I_1:=\int_0^T\!\!\!\int_\Omega\eta\,w(\widetilde q)\,\p_t^2v(q)\,\mathrm d\bm x\mathrm d t,\quad I_2:=\int_0^T\!\!\!\int_\Omega\mu\,w(\widetilde q)\,\mathrm d_{T-}^\alpha v(q)\,\mathrm d\bm x\mathrm d t,\\
I_3:=\int_0^T\!\!\!\int_\Omega w(\widetilde q)\,\cA^* v(q)\,\mathrm d\bm x\mathrm d t.
\end{gather*}
For $I_1$, by the usual integration by parts, we employ the initial condition of $w(\widetilde q)$ in \eqref{eq-gov-w} and the final condition of $v(q)$ in \eqref{eq-gov-v} to calculate
\begin{align}
I_1 & =\int_\Omega\eta\int_0^T w(\widetilde q)\,\p_t^2v(q)\,\mathrm d t\mathrm d\bm x\nonumber\\
& =\int_\Omega\eta\left\{\Big[w(\widetilde q)\,\p_t v(q)-v(q)\,\p_t w(\widetilde q)\Big]_0^T+\int_0^T v(q)\,\p_t^2w(\widetilde q)\,\mathrm d t\right\}\mathrm d\bm x\nonumber\\
& =\int_0^T\!\!\!\int_\Omega\eta\,v(q)\,\p_t^2w(\widetilde q)\,\mathrm d\bm x\mathrm d t.\label{eq-I1}
\end{align}
Similarly, we utilize \eqref{eq-adjoint} to treat $I_3$ as
\begin{equation}\label{eq-I3}
I_3=\int_0^T\left(w(\widetilde q),\cA^*v(q)\right)\mathrm d t=\int_0^T\left(\cA w(\widetilde q),v(q)\right)\mathrm d t=\int_0^T\!\!\!\int_\Omega v(q)\,\cA w(\widetilde q)\,\mathrm d\bm x\mathrm d t.
\end{equation}
Finally, for $I_2$, we take advantage of Lemma \ref{lem-FIbP} to deduce
\begin{align}
I_2 & =\int_\Omega\mu\int_0^T w(\widetilde q)\,\mathrm d_{T-}^\alpha v(q)\,\mathrm d t\mathrm d\bm x\nonumber\\
& =\int_\Omega\mu\left\{\Big[\left(J_{0+}^{2-\alpha}w(\widetilde q)\right)\p_t v(q)-v(q)\left(\p_t J_{0+}^{2-\alpha}w(\widetilde q)\right)\Big]_0^T+\int_0^T v(q)\,D_{0+}^\alpha w(\widetilde q)\,\mathrm d t\right\}\mathrm d\bm x.\label{eq-I20}
\end{align}
For the boundary term of $J_{0+}^{2-\alpha}w(\widetilde q)$ at $t=0$, we interpret it in the sense of limit and estimate
\begin{align*}
\left\|J_{0+}^{2-\alpha}w(\widetilde q)\right\|_{L^2(\Omega)} & \le\int_0^t\f{(t-s)^{1-\alpha}}{\Gamma(2-\alpha)}\|w(\widetilde q)(\,\cdot\,,s)\|_{L^2(\Omega)}\,\mathrm d s\\
& \le\sup_{0<s<t}\|w(\widetilde q)(\,\cdot\,,s)\|_{L^2(\Omega)}\f{t^{2-\alpha}}{\Gamma(3-\alpha)}\longrightarrow0\quad(t\to0).
\end{align*}
For $\p_t J_{0+}^{2-\alpha}w(\widetilde q)$, it follows from $w(\widetilde q)=0$ in $\Omega\times\{0\}$ that
\begin{align*}
\p_t J_{0+}^{2-\alpha}w(\widetilde q) & =\p_t\left(\int_0^t\f{s^{1-\alpha}}{\Gamma(2-\alpha)}w(\widetilde q)(\,\cdot\,,t-s)\,\mathrm d s\right)\\
& =\f{t^{1-\alpha}}{\Gamma(2-\alpha)}w(\widetilde q)(\,\cdot\,,0)+\int_0^t\f{s^{1-\alpha}}{\Gamma(2-\alpha)}\p_t w(\widetilde q)(\,\cdot\,,t-s)\,\mathrm d s\\
& =\int_0^t\f{(t-s)^{1-\alpha}}{\Gamma(2-\alpha)}\p_s w(\widetilde q)(\,\cdot\,,s)\,\mathrm d s.
\end{align*}
Hence, the same argument as above indicates
\[
\left\|\p_t J_{0+}^{2-\alpha}w(\widetilde q)\right\|_{L^2(\Omega)}\le\sup_{0<s<t}\|\p_s w(\widetilde q)(\,\cdot\,,s)\|_{L^2(\Omega)}\f{t^{2-\alpha}}{\Gamma(3-\alpha)}\longrightarrow0\quad(t\to0).
\]
Together with the final condition of $v(q)$, it turns out that all the boundary terms in \eqref{eq-I20} vanish and thus
\begin{equation}\label{eq-I2}
I_2=\int_0^T\!\!\!\int_\Omega\mu\,v(q)\,D_{0+}^\alpha w(\widetilde q)\,\mathrm d\bm x\mathrm d t.
\end{equation}
Substituting \eqref{eq-I1}, \eqref{eq-I2} and \eqref{eq-I3} into \eqref{eq-Gateaux2} and then into \eqref{eq-Gateaux1}, we recall the governing equation in \eqref{eq-gov-w} to conclude
\begin{align*}
\mathcal J'(q)\widetilde q & =\int_0^T\!\!\!\int_\Omega v(q)\,(\eta\,\p_t^2+\mu\,D_{0+}^\alpha+\cA)w(\widetilde q)\,\mathrm d\bm x\mathrm d t+\beta\int_\Omega q\,\widetilde q\,\mathrm d\bm x\\
& =\int_0^T\!\!\!\int_\Omega v(q)\,p\,\widetilde q\,\mathrm d\bm x\mathrm d t+\beta\int_\Omega q\,\widetilde q\,\mathrm d\bm x=\left(\int_0^T p(t)v(q)(\,\cdot\,,t)\,\mathrm d t+\beta\,q,\widetilde q\right).
\end{align*}
Since $\widetilde q\in L^2(\Omega)$ was chosen arbitrarily, we arrive at the desired result by means of the variational principle.
\end{proof}


\subsection{Finite element conjugate gradient method}

This subsection presents the complete numerical framework for reconstructing the spatial source term $q(\bm x)$ from partial interior measurements. Building upon the theoretical foundation established in Theorem \ref{thm:gradient}, we now develop a computationally efficient conjugate gradient algorithm that leverages the finite element discretization scheme introduced in Subsection \ref{sect:num}.

The gradient computation derived in Theorem \ref{thm:gradient} reveals a fundamental insight: the evaluation of $\mathcal{J}'(q)$ requires the solution of two coupled time-dependent partial differential equations, that is, the forward problem \eqref{main:e0} governing the viscoelastic membrane dynamics, and the adjoint problem \eqref{eq-gov-v} propagating the measurement misfit backward in time. Crucially, these two systems share identical mathematical structure, differing only in their temporal orientation and source terms. This structural symmetry permits the application of our previously developed finite element methodology to both equation systems, ensuring numerical consistency and computational efficiency.

The reconstruction algorithm employs an iterative optimization approach that successively refines the source term estimate by following descent directions in the cost functional landscape. Algorithm \ref{alg_inv} provides the complete computational procedure, which combines the finite element spatial discretization with an optimized conjugate gradient scheme for parameter space exploration.

\begin{algorithm}[htbp]
\caption{Finite Element Conjugate Gradient Method for Problem \ref{prob-ISP}}
\begin{algorithmic}[1]
\REQUIRE Measurement data $u_\delta(x,y,t)$ on $\overline\Omega_0 \times [0,T]$, initial guess $q^0 \in L^2(\Omega)$, regularization parameter $\beta > 0$, convergence tolerance $\epsilon > 0$
\ENSURE Reconstructed source term $q^*$ minimizing $\mathcal{J}(q)$

\STATE \textbf{Initialization:} Set iteration counter $k \leftarrow 0$, initial solution $q^0 \leftarrow q_{\text{init}}$
\STATE Compute initial cost $\mathcal{J}(q^0)$ by solving the forward problem \eqref{main:e0}

\WHILE{$k \leq K_{\text{max}}$}
\STATE \textbf{Step 1: Forward Problem Solution}
\STATE Solve the viscoelastic membrane equation \eqref{main:e0} with the source term $q^k$ to obtain state variable $u^k(\bm x,t)$
\STATE Employ the finite element discretization from Subsection \ref{sect:num} with the numerical scheme \eqref{Temp:e0}--\eqref{disc:frac1}

\STATE \textbf{Step 2: Adjoint Problem Solution}
\STATE Solve the backward problem \eqref{eq-gov-v} using the measurement residual $u^k - u_\delta$
\STATE \textbf{Step 3: Gradient Computation}
\STATE Evaluate the Fr\'echet derivative via Theorem \ref{thm:gradient}:
\STATE $\mathcal{J}'(q^k) \leftarrow \int_0^T p(t)v^k(\,\cdot\,,t)\,\mathrm d t + \beta q^k$

\STATE \textbf{Step 4: Convergence Check}
\IF{$\|\mathcal{J}'(q^k)\|_{L^2(\Omega)} < \epsilon$}
\STATE \textbf{break} \COMMENT{Optimality condition satisfied}
\ENDIF

\STATE \textbf{Step 5: Conjugate Direction Update}
\IF{$k = 0$}
\STATE $d^k \leftarrow -\mathcal{J}'(q^k)$ \COMMENT{Steepest descent initial step}
\ELSE
\STATE Compute Polak-Ribi\`ere parameter:
\STATE $\zeta^k \leftarrow \frac{\langle \mathcal{J}'(q^k), \mathcal{J}'(q^k) - \mathcal{J}'(q^{k-1}) \rangle}{\|\mathcal{J}'(q^{k-1})\|_{L^2(\Omega)}^2}$
\STATE Update descent direction: $d^k \leftarrow -\mathcal{J}'(q^k) + \zeta^k d^{k-1}$
\ENDIF

\STATE \textbf{Step 6: Optimal Step Size Selection}
\STATE Determine $\lambda^k > 0$ via Armijo line search condition:
\STATE $\mathcal{J}(q^k + \lambda^k d^k) \leq \mathcal{J}(q^k) + c_1 \lambda^k \langle \mathcal{J}'(q^k), d^k \rangle$
\STATE with backtracking to ensure sufficient decrease

\STATE \textbf{Step 7: Solution Update}
\STATE $q^{k+1} \leftarrow q^k + \lambda^k d^k$

\STATE \textbf{Step 8: Iteration Advance}
\STATE $k \leftarrow k + 1$
\ENDWHILE

\STATE \textbf{Return:} Optimal solution $q^* \leftarrow q^k$
\end{algorithmic}
\label{alg_inv}
\end{algorithm}

The algorithm exhibits several noteworthy features that ensure its effectiveness for the inverse problem under consideration:

\begin{itemize}
\item \textbf{Theoretical consistency:} The gradient computation rigorously follows the variational analysis presented in Theorem \ref{thm:gradient}, guaranteeing mathematical correctness.

\item \textbf{Numerical stability:} The identical treatment of forward and adjoint problems ensures that discretization errors remain consistent, preventing numerical artifacts in the reconstruction process.

\item \textbf{Adaptive regularization:} The Tikhonov term $\beta q$ naturally incorporated in the gradient provides inherent regularization, stabilizing the inversion against noise in the measurement data.
\end{itemize}

The algorithm terminates when the norm of the gradient falls below a specified tolerance $\epsilon$, indicating satisfaction of the first-order optimality conditions to within numerical precision. This criterion ensures that the reconstructed source term $q^*$ represents a genuine stationary point of the regularized cost functional, providing a mathematically sound solution to the inverse problem.


\section{Numerical validation and performance analysis}\label{sect:ex}

This section presents a comprehensive numerical investigation to validate the efficacy, accuracy, and robustness of the proposed finite element conjugate gradient method for solving the inverse source problem. Through systematically designed computational experiments, we assess the algorithm's performance under various challenging conditions, including measurement noise, limited observation domains, and different fractional orders. The numerical results demonstrate the method's practical applicability and provide insights into its convergence behavior and reconstruction capabilities.

\subsection{Robustness analysis under noisy measurements}\label{subsect:ex2}

To evaluate the algorithm's performance under realistic conditions, we consider measurement data contaminated with additive noise. The noisy observations are modeled as:
\begin{equation}\label{noise_model}
\begin{split}
u_\delta(\bm x,t) = u(\bm x,t) + \delta\,\mathrm{rand}(\bm x,t), \quad (\bm x,t)\in\Omega_0\times[0,T],
\end{split}
\end{equation}
where $\mathrm{rand}(\bm x,t)$ denotes uniformly distributed random noise on $\overline\Omega_0 \times [0,T]$. The reconstruction accuracy is quantified using the relative $L^2$ error:
\[
\mathcal{E}_{\mathrm{rel}} := \frac{\| q_{\mathrm{true}} - q_{\mathrm{rec}} \|_{L^2(\Omega)}}{\| q_{\mathrm{true}} \|_{L^2(\Omega)}},
\]
where $q_{\mathrm{true}} $ denotes the true solution and the reconstructed solution is $q_{\mathrm{rec}}$. Throughout this subsection, we fix the computational domain as $\Omega=(0,1)^2$, and simply set the coefficients in \eqref{main:e0} as
\[
\rho=\mu\equiv1,\quad\sigma\equiv0,\quad\cA=-\triangle=-(\p_x^2+\p_y^2).
\]


\subsubsection*{Experimental setup 1: baseline performance assessment}

We first establish a baseline configuration to evaluate the fundamental performance of our method. The temporal component of the source is specified as $p(t) = 2 + (2\pi t)^2$, while the true spatial source term is given by:
\begin{equation}\label{q_true}
q_{\text{true}}(x,y) = \frac{1}{2}\sin(\pi x)\cos(\pi y) + 1.
\end{equation}

The forward and adjoint problems are discretized using a $20\times20\times20$ mesh in space-time. The observable subdomain is set as $\Omega_0 = \Omega \setminus [0.05,0.95]^2$, with the final time $T=1.5$ and the noise level $\delta=1\%$.

Figure \ref{figure:3} presents the reconstruction results under these conditions. The top row compares the reconstructed solution with the true source term, demonstrating excellent visual agreement. The bottom row shows the absolute error distribution and the convergence history of the cost functional. Notably, the algorithm achieves a rapid convergence to $10^{-4}$ tolerance within a few iterations, highlighting its computational efficiency.

\begin{figure}[!htbp]
\includegraphics[width=\textwidth]{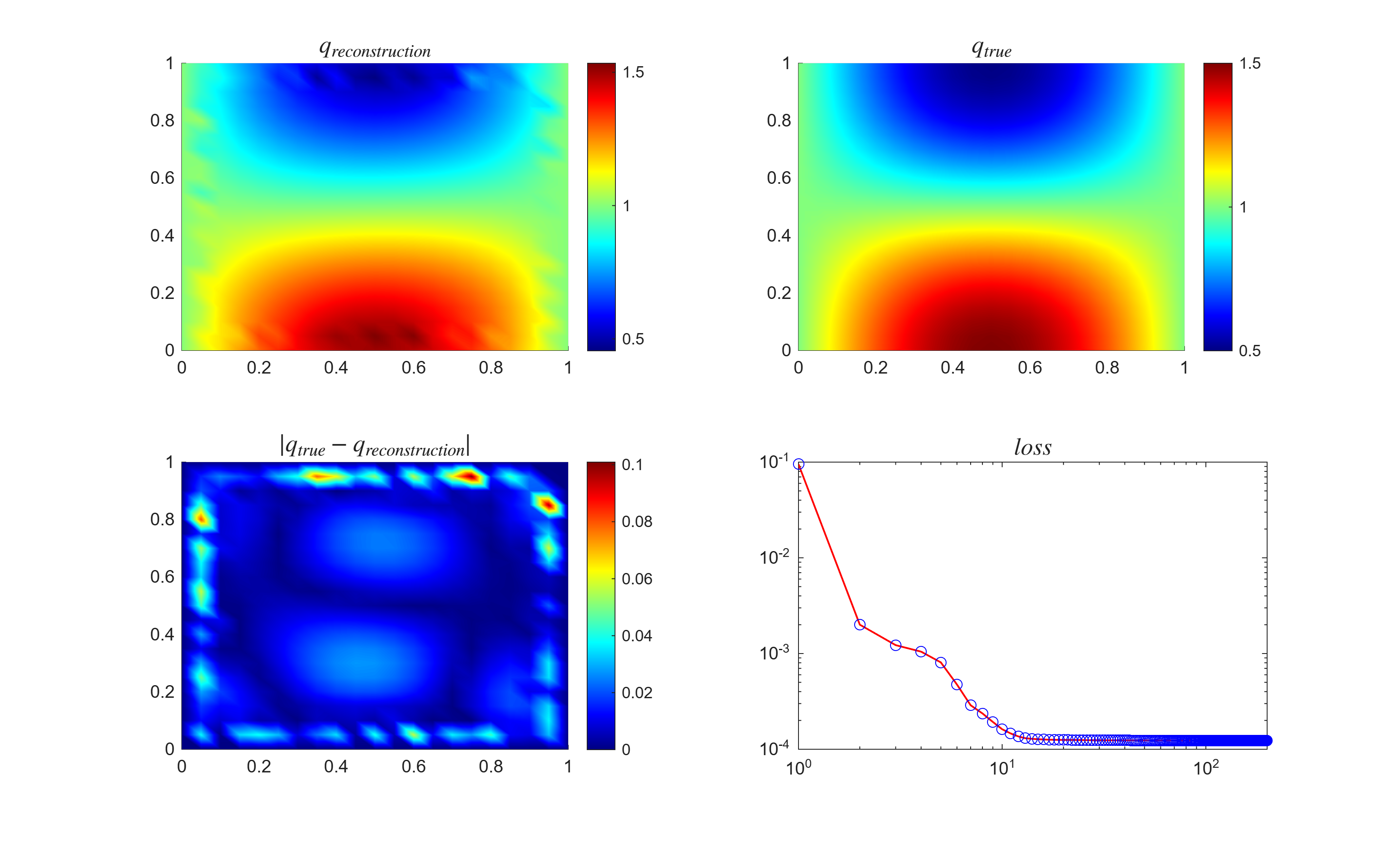}
\caption{Reconstruction results for $\alpha=1.5$ with $\Omega_0 = \Omega\setminus[0.05,0.95]^2$, $T=1.5$, and $\delta=1\%$ noise. Top: reconstructed (left) and true (right) solutions. Bottom: absolute error distribution (left) and the corresponding loss with respect to iterations (right).}\label{figure:3}
\end{figure}

To systematically evaluate robustness, we conduct parameter sensitivity studies summarized in Tables \ref{table:1} and \ref{table:2}. Table \ref{table:1} investigates the impact of noise levels and observation domain sizes, while Table \ref{table:2} examines the influence of fractional order variations.

\begin{table}[!htbp]\centering
\caption{Reconstruction accuracy under varying noise levels and observation domains ($\gamma=1.5$).}
\begin{tabular}{cccc}
\hline
Noise Level $\epsilon$ (\%) & Observation Domain $\Omega_0$   & Relative Error $\mathcal{E}_{\text{rel}}$ & Final Cost $\mathcal{J}$ \\
\hline
$1$ & $\Omega\setminus[0.1,0.9]^2$ & $2.45\times10^{-2}$ & $1.28\times10^{-4}$\\
$3$ & $\Omega\setminus[0.1,0.9]^2$ & $6.44\times10^{-2}$ & $1.89\times10^{-4}$\\
$5$ & $\Omega\setminus[0.1,0.9]^2$ & $9.05\times10^{-2}$ & $3.20\times10^{-4}$\\
\hline
$1$ & $\Omega\setminus[0.2,0.8]^2$ & $2.65\times10^{-2}$ & $1.30\times10^{-4}$\\
$1$ & $\Omega\setminus[0.1,0.9]^2$ & $2.45\times10^{-2}$ & $1.27\times10^{-4}$\\
$1$ & $\Omega\setminus[0.05,0.95]^2$ & $2.89\times10^{-2}$ & $1.25\times10^{-4}$\\
\hline
\end{tabular}\label{table:1}
\end{table}

\begin{table}[!htbp]\centering
\caption{Performance under different fractional orders ($\epsilon=2\%$, $\Omega_0=\Omega\setminus[0.05,0.95]^2$).}
\begin{tabular}{cccc}
\hline
Fractional Order $\gamma$ & Observation Domain $\Omega_0$ & Relative Error $\mathcal{E}_{\text{rel}}$ & Final Cost $\mathcal{J}$\\
\hline
$1.3$ & $\Omega\setminus[0.05,0.95]^2$ & $4.08\times10^{-2}$ & $1.41\times10^{-4}$\\
$1.6$ & $\Omega\setminus[0.05,0.95]^2$ & $3.99\times10^{-2}$ & $1.40\times10^{-4}$\\
$1.9$ & $\Omega\setminus[0.05,0.95]^2$ & $3.63\times10^{-2}$ & $1.40\times10^{-4}$\\
\hline
\end{tabular}\label{table:2}
\end{table}

The results demonstrate that our method maintains satisfactory reconstruction accuracy even under substantial noise contamination and limited observation data, confirming its robustness for practical applications.


\subsubsection*{Experimental setup 2: functional form sensitivity analysis}

We further investigate the algorithm's performance across different functional forms of the source term to assess its general applicability. Three distinct true solutions are considered:
\begin{equation}\label{diff_q}
\begin{split}
q_{\text{true}}^{(1)}(x,y) & = 3 - \exp\left(1 - \frac{x+y}{2}\right); \\
q_{\text{true}}^{(2)}(x,y) & = \frac{1}{2}\cos(\pi x)\cos(2\pi y) + 1; \\
q_{\text{true}}^{(3)}(x,y) & = \frac{1}{2}\cos(\pi x)\cos(\pi y) + 1.
\end{split}
\end{equation}

Figure \ref{figure:4} displays the reconstruction results for these test cases, while Table \ref{table:3} provides quantitative error metrics. The algorithm successfully recovers all three functional forms with consistent accuracy, demonstrating its versatility in handling diverse source term characteristics.

\begin{figure}[!htbp]
\includegraphics[width=\textwidth,height=.3\textwidth]{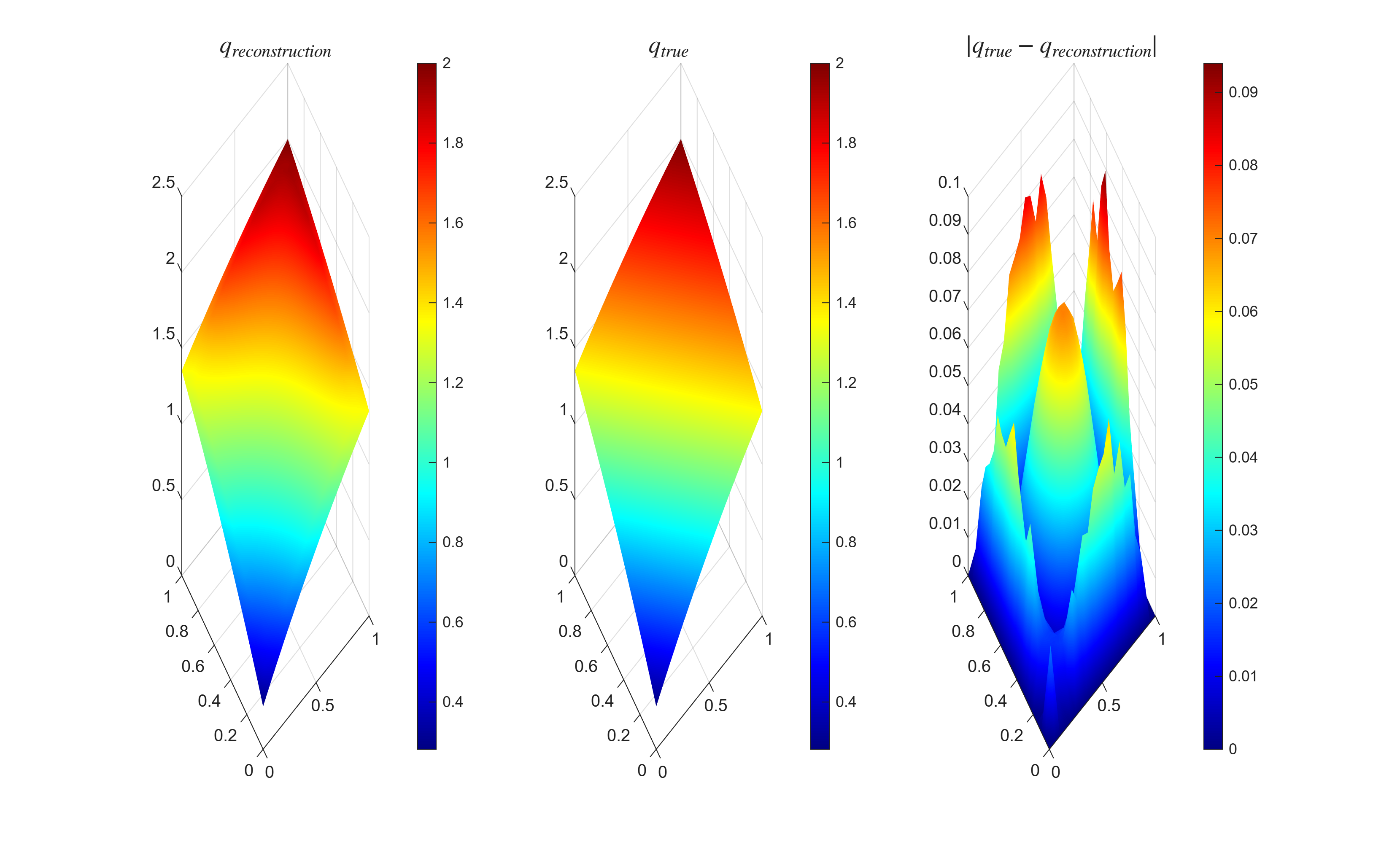}
\includegraphics[width=\textwidth,height=.3\textwidth]{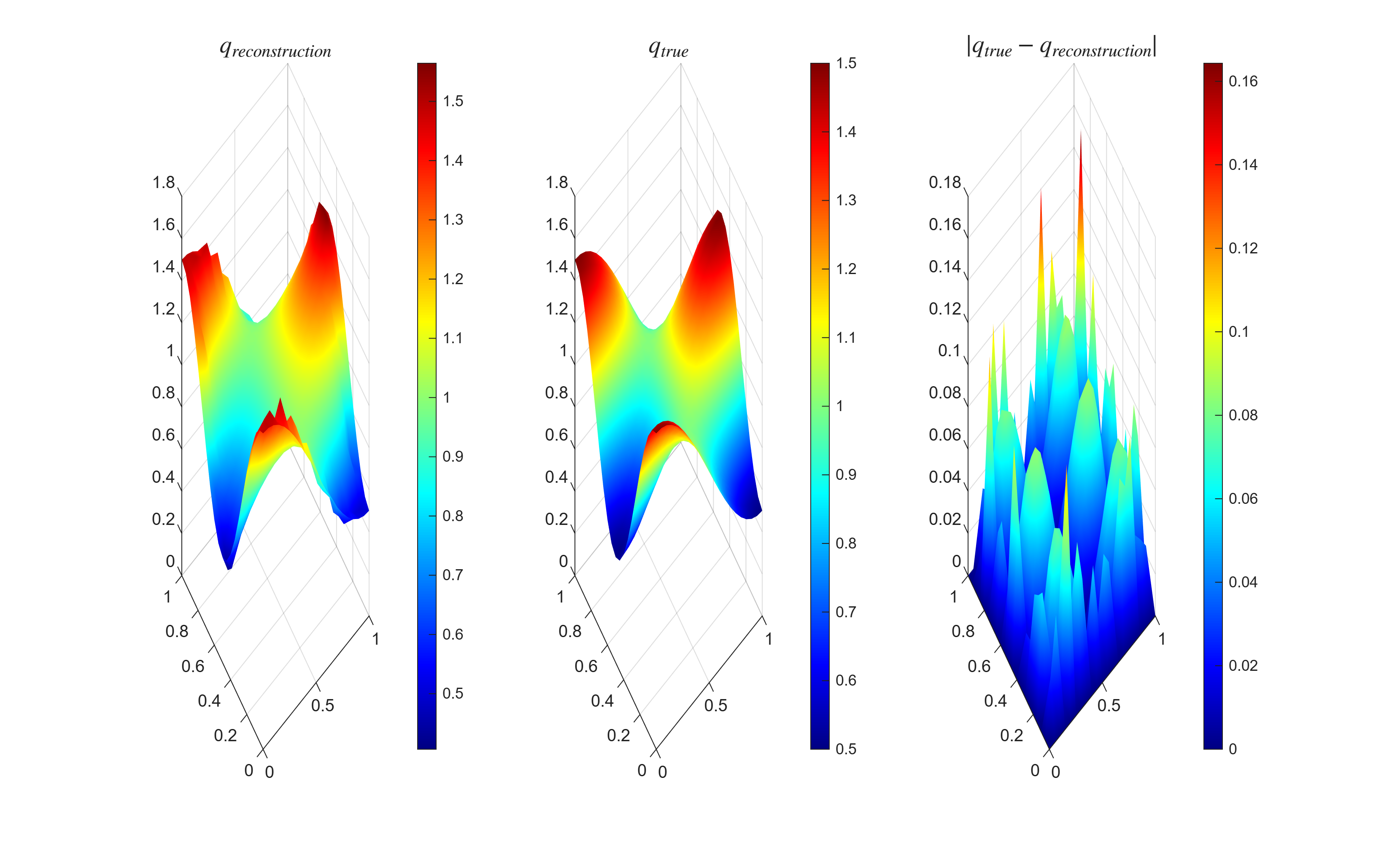}
\includegraphics[width=\textwidth,height=.3\textwidth]{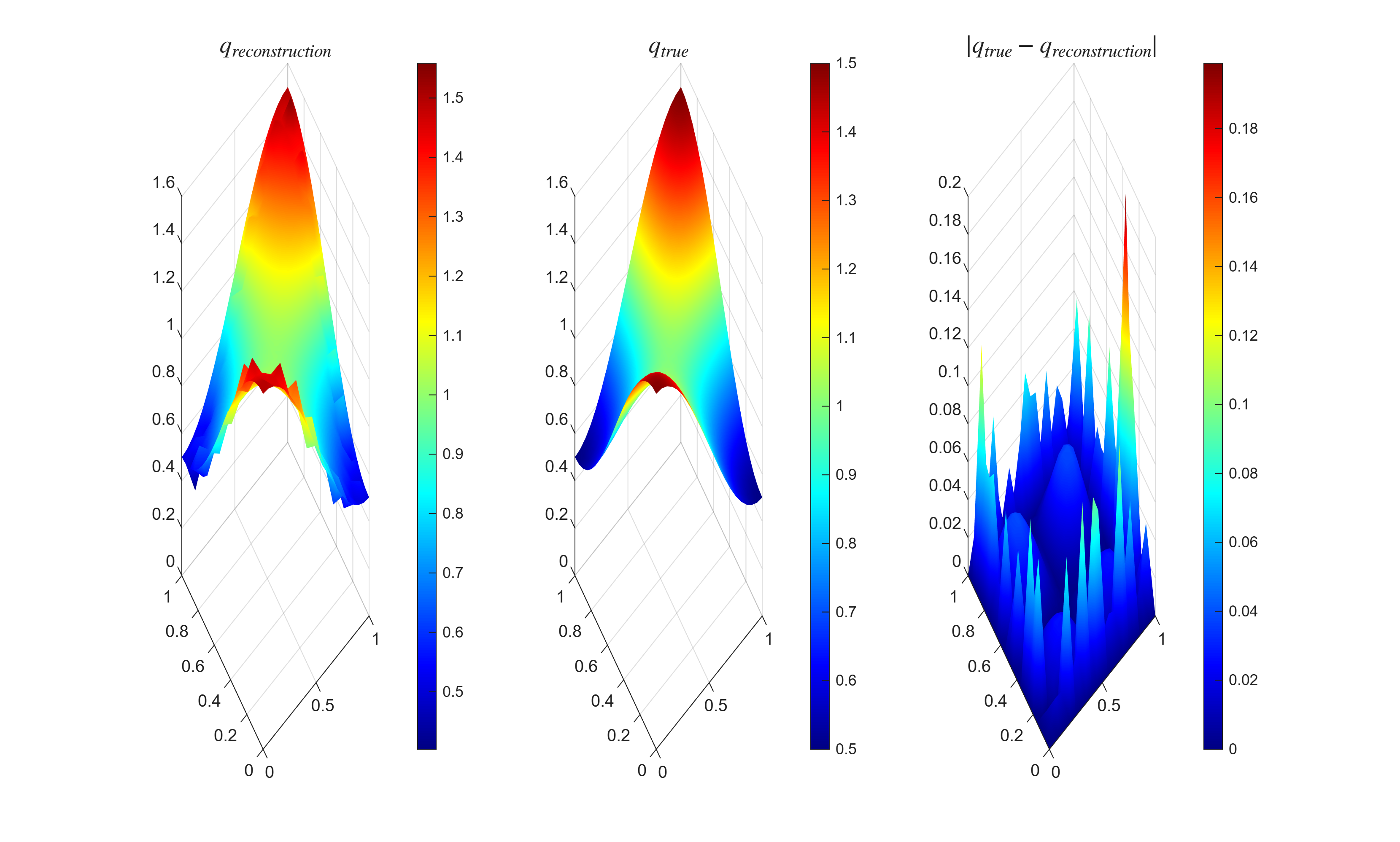}
\caption{Reconstruction results for three different source term functional forms ($\alpha=1.5$, $\Omega_0=\Omega\setminus[0.05,0.95]^2$, $\delta=1\%$). Each row shows reconstructed (left) and true (right) solutions for cases (1)--(3) respectively.}\label{figure:4}
\end{figure}

\begin{table}[!htbp]\centering
\caption{Performance across different source term functional forms ($\delta=1\%$, $\Omega_0=\Omega\setminus[0.05,0.95]^2$).}
\begin{tabular}{ccc}
\hline
Source Term $q_{\text{true}}(x,y)$ & Relative Error $\mathcal{E}_{\text{rel}}$ & Final Cost $\mathcal{J}$\\
\hline
$3-\exp(1-(x+y)/2)$ & $2.43\times10^{-2}$ & $2.83\times10^{-4}$\\
$\frac{1}{2}\cos(\pi x)\cos(2\pi y)+1$ & $4.12\times10^{-2}$ & $1.23\times10^{-4}$\\
$\frac{1}{2}\cos(\pi x)\cos(\pi y)+1$ & $2.75\times10^{-2}$ & $1.24\times10^{-4}$\\
\hline
\end{tabular}\label{table:3}
\end{table}


\subsection{Discussion of numerical findings}

The comprehensive numerical experiments reveal several key insights:

\begin{itemize}
\item \textbf{Noise robustness}: The method exhibits remarkable stability against measurement noise, with graceful degradation of reconstruction accuracy as noise levels increase.

\item \textbf{Partial data capability}: Satisfactory reconstructions are achieved even with limited observation domains, demonstrating the method's applicability to practical scenarios with restricted sensor coverage.

\item \textbf{Fractional order independence}: The algorithm's performance remains consistent across different fractional orders, confirming its robustness for various viscoelastic material models.

\item \textbf{Functional form versatility}: Successful reconstruction of diverse source term profiles highlights the method's general applicability beyond specific functional assumptions.
\end{itemize}

These results collectively establish the proposed method as a reliable and effective computational tool for solving inverse source problems in fractional viscoelastic membranes.


\section{Concluding remarks}\label{sect:end}

This paper has presented a comprehensive numerical framework for solving inverse source problems in viscoelastic membranes modeled by fractional derivatives. The core of our work establishes that the Riemann-Liouville fractional derivative provides the most physically consistent representation of viscoelastic damping effects, capturing the inherent memory and hereditary properties of real materials.

We have demonstrated that the Riemann-Liouville operator $D_{0+}^\alpha$ naturally arises from first principles when modeling viscoelastic damping, as it properly accounts for the entire deformation history of the material. This contrasts with the Caputo derivative, which requires artificial initial conditions that lack clear physical interpretation for intermediate fractional orders.


\subsection*{Future research directions}

Looking forward, several promising extensions of this work present exciting opportunities for both theoretical and applied research:

\textbf{Variable-order Fractional Calculus:}
A natural extension is to consider variable-order fractional derivatives where $\alpha = \alpha(\bm x,t)$, allowing the viscoelastic properties to vary spatially and temporally. This would enable modeling of heterogeneous materials and aging effects in viscoelastic structures.

\textbf{Multi-Scale Modeling Framework:}
Developing a multi-scale approach that couples macroscopic membrane dynamics with microscopic polymer chain interactions could provide deeper insights into the fundamental origins of fractional viscoelastic behavior. Such a framework might reveal universal scaling laws governing material response across different length scales.

\textbf{Data-Driven Fractional Modeling:}
Integrating machine learning techniques with our physics-based approach could enable automatic discovery of optimal fractional orders from experimental data, potentially revealing new relationships between material composition and fractional dynamics.

\textbf{Experimental Validation and Applications:}
The methodology developed here provides a powerful tool for practical applications including:
\begin{itemize}
\item Non-destructive testing of polymer composites and biological tissues
\item Structural health monitoring of viscoelastic damping systems
\item Material characterization through inverse analysis of dynamic responses
\end{itemize}

\section*{Acknowledgments}
The first author is supported by the National Natural Science Foundation of China (No. 12401555), the Natural Science Foundation of Shandong Province of China (No. ZR2025LZN010) and  the Taishan Scholars Program of Shandong Province (No. tsqn202507039). The second author is supported by JSPS KAKENHI Grant Numbers JP22K13954, JP23KK0049 and Guangdong Basic and Applied Basic Research Foundation (No. 2025A1515012248). 


\section*{AUTHOR DECLARATIONS}

\section*{Conflict of Interest}

No potential conflict of interest was reported by the authors.

\section*{Data availability statement}

All data generated or analyzed during this study are included in this article.


\end{document}